\documentclass[12pt,a4paper,reqno]{amsart}

\usepackage{amsmath,mathrsfs}
\usepackage{amssymb}
\usepackage{color}
\usepackage{dsfont}

\newcommand\NoBlackBoxes{\global\overfullrule0pt}
\NoBlackBoxes

\allowdisplaybreaks

\usepackage{todonotes}

\newcommand{\Bin}{\operatorname{Bin}\,}

\makeatletter
\let\serieslogo@\relax
\let\@setcopyright\relax

\newtheorem{definition}{Definition}[section]
\newtheorem{theorem}[definition]{Theorem}
\newtheorem{lemma}[definition]{Lemma}
\newtheorem{proposition}[definition]{Proposition}

\newtheorem{corollary}[definition]{Corollary}

\renewcommand{\P}{{\mathbb{P}}}
\newcommand{\E}{{\mathbb{E}}}

\newcommand{\R}{{\mathbb{R}}}

\renewcommand{\epsilon}{\varepsilon}
\renewcommand{\phi}{\varphi}

\newcommand{\be}{\begin{equation}}
\newcommand{\ee}{\end{equation}}
\newcommand{\bin}{\operatorname{Bin}}
\begin{document}

\title[Propagation of chaos in the critical Curie-Weiss model]{When does propagation of chaos in the critical Curie-Weiss model stop?}

\author[Nina Gantert]{Nina Gantert}
\address[Nina Gantert]{Technische Universit\"at M\"unchen, SoCIT, Department of Mathematics,
Boltzmannstr. 3,
85748 Garching bei M\"unchen,
Germany}

\email[Nina Gantert]{gantert@ma.tum.de}

\author[Matthias L\"owe]{Matthias L\"owe}
\address[Matthias L\"owe]{Fachbereich Mathematik und Informatik,
University of M\"unster,
Einsteinstra\ss e 62,
48149 M\"unster,
Germany}

\email[Matthias L\"owe]{maloewe@math.uni-muenster.de}

\author[Kilian Schnappinger]{Kilian Schnappinger}
\address[Kilian Schnappinger]{Technische Universit\"at M\"unchen, SoCIT, Department of Mathe\-matics,
Boltzmannstr. 3,
85748 Garching bei M\"unchen,
Germany}

\email[Kilian Schnappinger]{schnappingerkilian@gmail.com}


\date{\today}



\begin{abstract}
We study increasing propagation of chaos for the critical Curie-Weiss model (i.e.\ the mean-field Ising model at critical inverse temperature $\beta=1$, with no external field). We give a simple way to see that for windows of size $k(N)\ll \sqrt N$ we still have propagation of chaos (reproving earlier results, see e.g.\cite{BAZ_chaos}), while for $k(N)$ of order $\sqrt N$ the propagation of chaos breaks down. The law of a single spin converges to $\pi$, the Bernoulli law with parameter $1/2$. If  
$k(N) =\alpha \sqrt N$,
we give an explicit formula for the limiting distance in total variation of the law of the first $k(N)$ spins with respect to the $k$-fold product of $\pi$, as a function of $\alpha$. For even larger window sizes, the distribution of the spins has,  in the thermodynamical limit, maximal distance to the $k$-fold product of $\pi$. One of the ingredients of the proof is a result about the unimodality/non-unimodality of the law of the number of positive spins among the first $k=k(N)$ spins, which may be of independent interest.
\end{abstract}

\maketitle



\section{Introduction}
In statistical mechanics, the Curie-Weiss model serves as a mean-field description of ferromagnetic phenomena.
The model is characterized by a family of Gibbs probability measures $\mu_N$ defined on the configuration space $\{-1,+1\}^N$.
Two key parameters govern these measures: the inverse temperature $\beta>0$ and the external magnetic field strength $h\in\mathbb{R}$. For fixed values of $\beta>0$ and $h\in\mathbb{R}$, the Gibbs distribution is given by
\begin{equation} \label{eq:Gibbs}
\mu_N(\sigma):=\mu_{N,\beta,h}(\sigma):= \frac{1}{Z_N}\exp\left(\frac{\beta}{2N} \sum_{i,j=1}^N \sigma_i \sigma_j+h\sum_{i=1}^{N}\sigma_i\right), 
\end{equation}
where
\begin{equation}
\sigma= (\sigma_i)_{i=1}^N\in \{-1,+1\}^N\, .
\end{equation}
Here, $Z_N$ denotes the partition function, which serves as the normalization factor:
\begin{equation}
Z_N=Z_N(\beta,h)=\sum_{\sigma' \in \{-1,+1\}^N} \exp\left(\frac{\beta}{2N} \sum_{i,j=1}^N \sigma_i' \sigma_j'+h\sum_{i=1}^{N}\sigma_i'\right).
\end{equation}
Comprehensive treatments of the Curie-Weiss model and its principal asymptotic properties can be found in several textbooks~\cite{BovierSMoDS,EllisEntropyLargeDeviationsAndStatisticalMechanics,Velenik_book}, as well as in the research literature~\cite{Chatterjee_Shao,EL10,EllisNewman_80}. A central quantity in this model is the magnetization, which acts as the order parameter and is given by
\begin{equation*}
m_N:=m_N(\sigma) := \frac{1}{N} \sum_{i=1}^N \sigma_i=\frac{2\mathcal{P}_N}{N}-1,
\end{equation*}
where $\mathcal{P}_N:=\mathcal{P}_N(\sigma):=|\{i\in\{1,\ldots,N\}:\sigma_i=+1\}|$ counts the spins with positive orientation. We denote by $\mu_N\circ m_N^{-1}$ the induced probability distribution of the magnetization $m_N$ under $\mu_N$. The limiting behavior of the magnetization in the thermodynamic limit is characterized by
\begin{align}\label{eq:mconvergence}
\mu_N\circ m_N^{-1} \Rightarrow
\begin{cases}
\delta_{{\tt m}(\beta,h)},& \text{if}\quad h\neq 0\quad\text{or}\quad 0<\beta \le 1,\\ 
\frac{1}{2}\left(\delta_{{\tt m}(\beta,0)}+ \delta_{-{\tt m}(\beta,0)}\right),& \text{if}\quad h=0\quad\text{and}\quad \beta>1,
\end{cases}
\end{align}
revealing a phase transition at the critical inverse temperature $\beta=1$ when the external field vanishes ($h=0$). In this expression, $\Rightarrow$ signifies weak convergence, $\delta_x$ represents the Dirac delta measure concentrated at $x$, and ${\tt m}(\beta,h)$ is the solution of maximal absolute value of the equation
\begin{align}\label{eq:tanhh}
z=\tanh(\beta z+h).
\end{align}
More precisely: when $h>0$, there exists a unique positive solution; when $h<0$, a unique negative solution; and when $h=0$ with $0<\beta \leq 1$, the solution equals $0$. In the regime where $h=0$ and $\beta>1$, equation~\eqref{eq:tanhh} admits exactly two non-trivial solutions, namely ${\tt m}(\beta,0)$ and $-{\tt m}(\beta,0)$.
Another way to see that $\beta=1$ is the critical inverse temperature in the Curie-Weiss model is on the level of fluctuations of $m_N$. From now on, we fix
\[
h=0\, .
\]
Then, for $\beta <1$ the rescaled magnetization $\sqrt N m_N$ has a Gaussian limit. On the other hand, if $\beta=1$, we have to rescale $m_N$ by $N^{1/4}$ to obtain a non-standard limit law with density proportional to $e^{-\frac 1 {12} x^4}$.

A third, and closely related way to see the criticality of the model at $\beta=1$ is the so called increasing propagation of chaos result (cf. \cite{BAZ_chaos}).
Let us fix an arbitrary positive integer $k\in\mathbb{N}$ and select any collection of $k$ spins from the total of $N$ available spins. Due to the exchangeability of the spin configuration $(\sigma_i)_{i=1}^N$ with respect to the Gibbs measure $\mu_N$, without loss of generality we can focus our attention on the initial $k$ spins and analyze their joint marginal law $\mu^{(k)}_{N}= \mu^{(k)}_{N,\beta}$.
The propagation of chaos property for Gibbs measures, asserts that the joint distribution of any finite subset of $k$ spins converges to a product measure as the system size grows. More precisely,
with 
\begin{equation}
\pi: = \frac{1}{2}\delta_{-1}+\frac{1}{2}\delta_{+1}
\end{equation}\label{pi-def}
we have
\begin{equation}\label{prop-chaos}
d_{\mathrm{TV}}\Big(\mu_N^{(k)},\pi^{\otimes k}\Big)\to 0\, .
\end{equation}
The statement in \eqref{prop-chaos} is often referred to as \lq\lq propagation of chaos\rq\rq.
Our main goal will be to investigate if \eqref{prop-chaos} still holds true for $k=k(N)$ depending on $N$, which is referred to as \lq\lq increasing propagation of chaos\rq\rq.
We introduce the counting variable $\mathcal{P}_{k}:=|\{j\in\{1, \ldots,k\}:\sigma_j=+1\}|$, which counts the number of up-spins within the selected subset. Observe that
the quantity $\mathcal{P}_{k}$ provides a complete characterization of $\mu^{(k)}_{N,\beta}$, which allows us to equivalently analyze the law of $\mathcal{P}_{k}$
under $\mu_N$, denoted by $\mu_N\circ \mathcal{P}_{k}^{-1}$. More precisely,
\begin{equation}\label{meas-equiv-numb}
d_{\mathrm{TV}}\Big(\mu_N^{(k)},\pi^{\otimes k}\Big) = 
d_{\mathrm{TV}}\big(\mu_N\circ\mathcal P_{k}^{-1},\bin \big(k,\frac{1}{2}\big)\big)\, .
\end{equation}
Now, indeed in the Curie-Weiss model (and many other mean-field models) propagation of chaos holds in a  stronger sense: consider the situation, when $k=k(N)$ depends on $N$. If $k(N)\ll N$ (i.e.\ $\tfrac{k(N)}{N}\to 0$ for $N \to \infty$) and $\beta<1$, then still $\mu_N \circ \mathcal{P}_{k(N)}^{-1}$ is well approximated by a $\Bin(k(N), \frac 12)$-distribution. Similarly, if $k(N) \ll N$ and $\beta>1$, 
\begin{equation*}
\lim_{N\to\infty}d_{\mathrm{TV}}\Big(\mu_N\circ\mathcal P_{k(N)}^{-1},\frac{1}{2}\Bin\Big(k(N),\frac{1+{\tt m}(\beta,0)}{2}\Big)+\frac{1}{2}\Bin\Big(k(N),\frac{1-{\tt m}(\beta,0)}{2}\Big)\Big)=0.
\end{equation*}
Finally, for $\beta=1$ and $k(N) \ll \sqrt N$ we can again approximate
the distribution of $\mathcal{P}_{k(N)}$ by a $\Bin(k(N), \frac 12)$-distribution. All of this has implicitly or 
explicitly been shown in \cite{BAZ_chaos}. In \cite{JKLM}, the authors show that the bound $k(N)\ll N$ is sharp, i.e.\ that propagation of chaos breaks down for larger block sizes. The case $\beta=1$ is not treated there. 

The aim of the present paper is to close this gap. We will show increasing propagation of chaos in the Curie-Weiss model for $\beta=1$ and $k(N) \ll \sqrt N$ by elementary methods and prove that at $k(N)=\alpha N^{\frac 12}$ with $\alpha>0$ propagation of chaos breaks down (an idea why this might be true is already contained in \cite{JKL25}). We will also see that for  
$k(N)=\alpha N^{\gamma}$ with $\alpha>0, \gamma >\frac 12$ the distributions of $\mathcal{P}_{k(N)}$ and $\mathrm{Bin}(k(N), \frac 12)$ have maximal total variation distance in the limit $N \to \infty$. 
More precisely, we will show the following.
\begin{theorem}\label{thm:mainbefore}
Assume $h=0$ and $\beta=1$, and let $k=k(N)\leq N$ with $k(N)\to\infty$.
Then the following statements hold:
\begin{enumerate}
\item If $k(N)/N^{1/2}\to0$, then
$d_{\mathrm{TV}}\Big(\mu_N^{(k)},\pi^{\otimes k}\Big)\to 0$.
\item If $k(N)/N^{1/2}\to\alpha\in(0,\infty)$, then
\begin{equation*}
u(\alpha) : = \lim\limits_{N\to\infty}d_{\mathrm{TV}}\Big(\mu_N^{(k)},\pi^{\otimes k}\Big)  \text{ exists and } 0 < u(\alpha) < 1.
\end{equation*}
An explicit formula for $u(\alpha)$ is given in \eqref{formula-distance}.
\item If $k(N)/N^{1/2}\to\infty$, then
$
d_{\mathrm{TV}}\Big(\mu_N^{(k)},\pi^{\otimes k}\Big)\to1.
$
\end{enumerate}
\end{theorem}
Due to the preceding discussion, see \eqref{meas-equiv-numb}, Theorem \ref{thm:mainbefore} is equivalent to Theorem \ref{thm:main1} below and it suffices to prove the latter.
\begin{theorem}\label{thm:main1}
Assume $h=0$ and $\beta=1$, and let $k=k(N)\leq N$ with $k(N)\to\infty$.
Then the following statements hold:
\begin{enumerate}
\item If $k(N)/N^{1/2}\to0$, then
$d_{\mathrm{TV}}\left(\mu_N\circ\mathcal P_{k(N)}^{-1},
\bin \left(k(N),\frac{1}{2}\right)\right)\to 0$.
\item If $k(N)/N^{1/2}\to\alpha\in(0,\infty)$, then
\begin{equation*}
u(\alpha) = \lim\limits_{N\to\infty}d_{\mathrm{TV}}\big(\mu_N\circ\mathcal P_{k(N)}^{-1},\bin \big(k(N),\frac{1}{2}\big)\big) \text{ exists and } 0 < u(\alpha) < 1.
\end{equation*}
An explicit formula for $u(\alpha)$ is given in \eqref{formula-distance}.
\item If $k(N)/N^{1/2}\to\infty$, then
$
d_{\mathrm{TV}}\left(\mu_N\circ\mathcal P_{k(N)}^{-1},
\bin \left(k(N),\frac{1}{2}\right)\right)\to1.
$
\end{enumerate}
\end{theorem}

The proof of Theorem~\ref{thm:main1} relies on an exact conditional
representation of the Curie--Weiss marginals in terms of the empirical
proportion of positive spins, together with the critical scaling of this random parameter.
In Section~2 we recall the exchangeability of the Curie--Weiss spins and
introduce the corresponding mixing variable $U_N$. We show that at criticality
($\beta=1$, $h=0$) the fluctuations of $U_N$ around $1/2$ occur on the scale
$N^{-1/4}$ and converge in distribution to a non-degenerate quartic law.

The threshold $k\asymp N^{1/2}$ results from a comparison of two fluctuation
scales. Conditionally on $U_N$, the number $\mathcal P_k$ of positive spins in
the subsystem has mean $kU_N$ and, in the range relevant for the threshold,
fluctuations of order $\sqrt{k}$. Since
\[
U_N-\frac12=O_{\mathbb P}(N^{-1/4}),
\]
the critical fluctuations of $U_N$ produce a random shift of the conditional
mean of order $kN^{-1/4}$. The ratio between this shift and the intrinsic
sampling fluctuations is therefore
\[
\frac{kN^{-1/4}}{\sqrt{k}}
=\sqrt{k}\,N^{-1/4}
=\left(\frac{k}{N^{1/2}}\right)^{1/2}.
\]
Thus the random shift is asymptotically invisible for $k\ll N^{1/2}$, remains
of order one for $k\asymp N^{1/2}$, and dominates for $k\gg N^{1/2}$. These
three cases correspond, respectively, to propagation of chaos, its breakdown
at the critical scale, and asymptotic separation from the product measure.

These three regimes are treated in Sections~3 and~4, while the necessary
estimates are collected in Section~2.
For the proof of the result for the critical scale, we need a statement about the unimodality of the law of the number of positive spins among the first $k=k(N)$ spins. This may be of independent interest and we give a more general result in Corollary \ref{cor:unimodality-transition} in the appendix.

\section{Structural facts about the Curie--Weiss model}

The dependence structure of the Curie--Weiss model can be described in terms
of the empirical proportion of positive spins. More precisely, conditionally
on the total number of positive spins, every configuration with this number of
positive spins is equally likely. Consequently, the number of positive spins
observed in a subsystem is hypergeometrically distributed.

This exact conditional representation allows us to separate the two sources
of randomness relevant for propagation of chaos: the fluctuations of the
empirical magnetization of the whole system and the randomness arising from sampling a
subsystem. For subsystems of size \(k=o(\sqrt N)\), sampling without
replacement can furthermore be approximated in total variation by independent
Bernoulli sampling. This will yield a binomial-mixture representation in the
regime in which we prove propagation of chaos.

Let
\[
\mathcal P_N:=\sum_{i=1}^N\mathbf 1_{\{\sigma_i=+1\}}
\qquad \text{and} \quad 
U_N:=\frac{\mathcal P_N}{N}=\frac{1+m_N}{2}.
\]
Since the Curie--Weiss Hamiltonian depends on the spin configuration only through the total magnetization, conditionally on $\mathcal P_N=s$, the configuration is uniformly distributed over the
$\binom Ns$ configurations containing exactly $s$ positive spins.
Therefore, for every $k\le N$,
\[
\mathcal P_k\mid \mathcal P_N=s \sim \operatorname{Hyp}(N,s,k).
\]
In other words the distribution of the positive spins under the Gibbs measure can be written as
\begin{equation}\label{eq:hypergeometric-mixture}
\mu_N\circ\mathcal P_k^{-1}=\mathbb E\!\left[
\operatorname{Hyp}(N,NU_N,k)\right].
\end{equation}
The representation \eqref{eq:hypergeometric-mixture} separates the two sources
of randomness that are relevant for the behaviour of the subsystem.
Conditionally on $U_N$, the distribution of $\mathcal P_k$ is 
hypergeometric, while the effect of the Curie--Weiss interaction is encoded in
the fluctuations of the random parameter $U_N$. We therefore next determine
the critical scale of these fluctuations.

We first prove

\begin{proposition}\label{prop:critical-scaling}
Let $\beta=1$ and $h=0$. Then
$N^{1/4}\left(U_N-\frac{1}{2}\right)\Rightarrow X$,
where $X$ has density
\[
f_X(x)=\frac{\exp\left(-\frac{4}{3}x^4\right)}
{\displaystyle\int_{\mathbb R}
\exp\left(-\frac{4}{3}y^4\right)\,dy},
\qquad x\in\mathbb R.
\]
\end{proposition}

\begin{proof}
The classical limit theorem for the Curie--Weiss model at $\beta=1$ and $h=0$ states that
$N^{1/4}m_N\Rightarrow M$,
where $M$ has a density proportional to
$\exp\left(-\frac{x^4}{12}\right)$.
Since
$N^{1/4}\left(U_N-\frac{1}{2}\right)=
\frac{1}{2}N^{1/4}m_N$,
the continuous mapping theorem yields
\[
N^{1/4}\left(U_N-\frac{1}{2}\right)
\Rightarrow \frac{M}{2}.
\]
Writing $X=M/2$, we see that the density of $X$
is proportional to
$\exp\left(-\frac{(2x)^4}{12}\right)=
\exp\left(-\frac{4}{3}x^4\right)$.
\end{proof}

For the proof of propagation of chaos, convergence in distribution alone is
not sufficient. We will also need control of the first moment of the critical
fluctuations. The following moment estimate provides the required uniform
integrability.

\begin{lemma}\label{lem:critical-moments}
For every $p\geq 1$,
$\sup_{N\in\mathbb N}\mathbb E_{\mu_N}\big[|N^{1/4}m_N|^p
\big]<\infty$
and thus,
\[
\sup_{N\in\mathbb N}\mathbb E_{\mu_N}
\Big[\big|N^{1/4}(U_N-\frac{1}{2})\big|^p\Big]<\infty.
\]
In particular,
\[
N^{1/4}\mathbb E_{\mu_N}\left[|U_N-\frac{1}{2}|\right]
\longrightarrow\mathbb E[|X|]
\]
and
$\mathbb E_{\mu_N}\left[|U_N-\frac{1}{2}|\right]=O\left(N^{-1/4}\right)$.
\end{lemma}
\begin{proof}
Let $\mathcal M_N=\left\{-1,-1+\frac{1}{N},\ldots,1-\frac{1}{N},1
\right\}$
denote the set of possible values of the empirical magnetization. At the
critical point $\beta=1$ and $h=0$, its distribution is given by
\[
\mathbb P_{\mu_N}(m_N=m)
=\frac{1}{\mathcal Z_N}\binom{N}{\frac{N(1+m)}{2}}
\exp\left(\frac{N}{2}m^2\right),\qquad m\in\mathcal M_N,
\]
where
\[
\mathcal Z_N=\sum_{m\in\mathcal M_N}\binom{N}{\frac{N(1+m)}{2}}
\exp\left(\frac{N}{2}m^2\right).
\]

Define
\[
I(m)=\frac{1+m}{2}\log(1+m)
+\frac{1-m}{2}\log(1-m), \qquad m\in[-1,1],
\]
(with the convention $0\log 0=0$) and set
$$K(m):=I(m)-\frac{m^2}{2}.$$
The function $K$ is even and satisfies
$K(0)=K'(0)=0$
and
\[
K''(m)=\frac{1}{1-m^2}-1=\frac{m^2}{1-m^2},
\qquad |m|<1,
\]
in particular, $K''(m)\ge m^2$.
By integration, we obtain
$K(m)\ge \frac{m^4}{12}$ for $m\in[-1,1]$.
Moreover, if $|m|\le 1/2$, then
$K''(m)\le\frac{4}{3}m^2$,
and hence $K(m)\le \frac{m^4}{9}$.

Let us derive a lower bound on $\mathcal Z_N$. Stirling's formula yields that there exist constants $c_0,C_0>0$ such that
\[
c_0N^{-1/2}e^{NH(q)} \le\binom{N}{Nq}
\le C_0N^{-1/2}e^{NH(q)}
\]
whenever $q\in[1/4,3/4]$ and $Nq\in\mathbb N$, where
\[
H(q)=-q\log q-(1-q)\log(1-q).
\]
Since
$H(\frac{1+m}{2})=\log 2-I(m)$,
it follows that, uniformly for $|m|\le 1/2$,
\[
c_0 2^N N^{-1/2}e^{-NK(m)}
\le
\binom{N}{\frac{N(1+m)}{2}}
\exp\left(\frac{N}{2}m^2\right)
\le
C_0 2^N N^{-1/2}e^{-NK(m)}.
\]

There are at least $c_1N^{3/4}$ points $m\in\mathcal M_N$ satisfying
$|m|\le N^{-1/4}$. For such $m$,
\[
NK(m)\le\frac{N}{9}m^4\le\frac{1}{9}.
\]
Hence
\[
\begin{aligned}
\mathcal Z_N&\ge
\sum_{\substack{m\in\mathcal M_N\\ |m|\le N^{-1/4}}} \binom{N}{\frac{N(1+m)}{2}} \exp\left(\frac{N}{2}m^2\right)
\ge c_2 2^N N^{-1/2}N^{3/4}=c_2 2^N N^{1/4}.
\end{aligned}
\]

Combining this lower bound with the upper Stirling estimate and the inequality
$K(m)\ge m^4/12$, we obtain, uniformly for $|m|\le 1/2$,
\begin{equation}\label{UB-m-small}
\mathbb P_{\mu_N}(m_N=m)
\le C_1N^{-3/4} \exp\left(-\frac{N}{12}m^4\right).
\end{equation}
It remains to control the region $|m|>1/2$. Since $K$ is continuous,
$K(0)=0$, and $K(m)>0$ for $m\neq0$, there exists $c_3>0$ such that $K(m)\geq c_3$ whenever $|m|\geq\frac12$.
Writing $q=(1+m)/2$ and using
$\binom{N}{Nq}\leq e^{NH(q)}$,
together with the lower bound
$\mathcal Z_N\geq c_2 2^N N^{1/4}$, we obtain
\[
\P_{\mu_N}(m_N=m) \leq C_1 N^{-1/4}e^{-NK(m)} \leq C_1 N^{-1/4}e^{-c_3N}
\]
for $|m|>1/2$. Since there are at most $N+1$ possible values of $m_N$,
\[
\P_{\mu_N}\left(|m_N|>\frac12\right) \leq C_1 N^{3/4}e^{-c_3N}
\leq C_2e^{-c_4N}
\]
for suitable constants $C_2,c_4>0$.
Set $Y_N:=N^{1/4}|m_N|$.
For every integer $j\ge 0$, the number of points $m\in\mathcal M_N$ satisfying $j\le N^{1/4}|m|<j+1$
is bounded by $C_3N^{3/4}$. Hence, by \eqref{UB-m-small},
\[
\mathbb P_{\mu_N}\left(j\le Y_N<j+1,\ |m_N|\le\frac{1}{2}
\right) \le C_4e^{-c_5j^4}.
\]
Together with the exponentially small contribution from $|m_N|>1/2$, this
implies that there exist constants $c,C>0$ such that
\[
\mathbb P_{\mu_N}(Y_N\ge t)
\le
Ce^{-ct^4}
\]
for every $t\ge0$, uniformly in $N$.

Finally, for every $p\ge1$, integration by parts gives
$\mathbb E[Y^p]= p\int_0^\infty t^{p-1}\mathbb P(Y\ge t)\,dt$
(for non-negative random variables $Y$) and thus
\[
\mathbb E_{\mu_N}[Y_N^p]=p\int_0^\infty t^{p-1}\mathbb P_{\mu_N}(Y_N\ge t)\,dt\le Cp\int_0^\infty t^{p-1}e^{-ct^4}\,dt <
\infty,
\]
uniformly in $N$. Therefore, $\sup_{N\in\mathbb N}
\mathbb E_{\mu_N} [|N^{1/4}m_N |^p]<\infty.$
Since $ N^{1/4}(U_N-\frac{1}{2})=
\frac{1}{2}N^{1/4}m_N$,
the corresponding moment bound for $U_N$ follows. In particular,
the sequence
$\{N^{1/4}|U_N-\frac{1}{2}|\}_{N\in\mathbb N}$
is uniformly integrable. Together with Proposition~\ref{prop:critical-scaling} this
implies
\[
N^{1/4}\mathbb E_{\mu_N}\left[\big|U_N-\frac{1}{2}\big|\right]
\longrightarrow
\mathbb E[|X|], 
\]
hence
$\mathbb E_{\mu_N}\left[ |U_N-\frac{1}{2}|\right]
=O\left(N^{-1/4}\right)$.
\end{proof}

\section{Propagation of chaos below the critical scale}
\label{sec:propagation}
In this section we prove part~(1) of Theorem~\ref{thm:main1}. 
We first recall a well known approximation for sampling without replacement by sampling with
replacement.

\begin{lemma}\label{lem:hyp-bin}
For every $N\in\mathbb N$, every $k\leq N$, and every
$s\in\{0,\ldots,N\}$,
\[
d_{\mathrm{TV}}\left(\operatorname{Hyp}(N,s,k),\bin \left(k,\frac{s}{N}\right)\right)
\leq\frac{4k }{N}.
\]
In particular, the estimate is uniform in $s$.
\end{lemma}

\begin{proof}
This is the classical total variation bound for sampling with and without
replacement; see, for example, \cite[Theorem~4]{DiaconisFreedman1980} or \cite[Theorem~2]{Ehm1991}.
\end{proof}

Applying Lemma~\ref{lem:hyp-bin} conditionally on $U_N$ and using
\eqref{eq:hypergeometric-mixture}, we obtain
\[
d_{\mathrm{TV}}\left(\mu_N\circ\mathcal P_k^{-1},
\mathbb E\left[\bin (k,U_N)\right]\right)
\leq \frac{4k}{N}.
\]

We next control the effect of replacing the random parameter $U_N$ by
$1/2$.

\begin{lemma}\label{lem:binomial-parameter}
For every $k\in\mathbb N$ and every $p\in[0,1]$,
\[
d_{\mathrm{TV}}\Big(\bin (k,p),\bin \Big(k,\frac{1}{2}\Big)\Big)
\leq \sqrt{2k }\Big|p-\frac{1}{2}\Big|.
\]
\end{lemma}

\begin{proof}
Let $\operatorname{Ber}(p)$ denote the Bernoulli distribution with parameter $p$. Since the binomial distribution is the image of the corresponding product measure under the summation map, we obtain
\[
d_{\mathrm{TV}}\Big(\bin (k,p),\bin \Big(k,\frac{1}{2}\Big)\Big)\leq d_{\mathrm{TV}}\Big( \operatorname{Ber}(p)^{\otimes k},
\operatorname{Ber}\Big(\frac{1}{2}\Big)^{\otimes k} \Big).
\]
Pinsker's inequality and tensorization of relative entropy yield
\[
\begin{aligned}
d_{\mathrm{TV}}\Big(\operatorname{Ber}(p)^{\otimes k},\operatorname{Ber}\Big(\frac{1}{2}\Big)^{\otimes k}\Big)
&\leq\sqrt{\frac{k}{2}H\Big(\operatorname{Ber}(p)\mid\operatorname{Ber}\Big(\frac{1}{2}\Big)\Big)}.
\end{aligned}
\]
Hence,
\[
H\Big(\operatorname{Ber}(p)\mid \operatorname{Ber}\big(\frac{1}{2}\big)\Big)
= p\log(2p) + (1-p) \log(2(1-p)))\leq (2p-1)^2 = 4\big(p-\frac{1}{2}\big)^2
\]
where we used the inequality $\log x\leq x-1$.
Combining the preceding estimates proves the claim.
\end{proof}

We can now prove propagation of chaos below the critical scale.

\begin{proof}[Proof of Theorem~\ref{thm:main1}, part~(1)]
Recall from \eqref{eq:hypergeometric-mixture} that $\mu_N\circ\mathcal P_k^{-1}=\mathbb E\big[\operatorname{Hyp}(N,NU_N,k)\big]$.
Hence, by the triangle inequality,
\begin{multline*}
d_{\mathrm{TV}}\Big(\mu_N\circ\mathcal P_k^{-1},\bin \big(k,\frac12\big)
\Big) \leq
d_{\mathrm{TV}}\Big(
\mathbb E\big[\operatorname{Hyp}(N,NU_N,k)\big],
\mathbb E\big[\operatorname{Bin}(k,U_N)\big]
\Big)
\\+d_{\mathrm{TV}}\Big(\mathbb E\big[\operatorname{Bin}(k,U_N)\big],
\bin\big(k,\frac12\big)\Big).
\end{multline*}
By the joint convexity of total variation and
Lemma~\ref{lem:hyp-bin},
\begin{eqnarray*}
d_{\mathrm{TV}}\left(
\mathbb E\big[\operatorname{Hyp}(N,NU_N,k)\big],
\mathbb E\big[\bin(k,U_N)\big]
\right)
\leq
\mathbb E\left[
d_{\mathrm{TV}}\left(\operatorname{Hyp}(N,NU_N,k),
\Bin(k,U_N)\right)
\right]\cr
\leq \frac{4k}{N}.
\end{eqnarray*}
Again, by the convexity of total variation and Lemma~\ref{lem:binomial-parameter},
\begin{eqnarray*}
d_{\mathrm{TV}}\Big(\mathbb E\Big[\bin (k,U_N)\Big],
\bin \Big(k,\frac{1}{2}\Big)
\Big)
&\leq&
 \mathbb E\Big[d_{\mathrm{TV}}\Big(
\bin (k,U_N),\bin \Big(k,\frac{1}{2}\Big)
\Big) \Big]\cr
&\leq& \sqrt{2k }\,\mathbb E\left[\Big|U_N-\frac{1}{2}\Big|\right].
\end{eqnarray*}
Lemma~\ref{lem:critical-moments} gives
$\mathbb E\left[\Big|U_N-\frac{1}{2}\Big|\right]=O(N^{-1/4})$.
Therefore, $$d_{\mathrm{TV}}\left(\mu_N\circ\mathcal P_{k(N)}^{-1},
\bin \left(k(N),\frac{1}{2}\right)\right)   
\leq\frac{4k }{N}+
O\Big(\sqrt{k}\,N^{-1/4}\Big).$$

If $k=k(N)=o(N^{1/2})$, then $\frac{k}{N}\to 0$ as well as
and $\sqrt{k}\,N^{-1/4}=\Big(\frac{k}{N^{1/2}}\Big)^{1/2}
\to 0$.
Thus $d_{\mathrm{TV}}\left(\mu_N\circ\mathcal P_{k(N)}^{-1},
\bin \left(k(N),\frac{1}{2}\right)\right)   
\to 0$,
which proves part~(1) of Theorem \ref{thm:main1}.
\end{proof}

\section{Breakdown of propagation of chaos at criticality}
\label{sec:breakdown}

In this section we prove parts~(2) and~(3) of
Theorem~\ref{thm:main1}. At the critical scale
$k(N)\asymp N^{1/2}$, the fluctuations of the empirical
magnetization of the whole system are of the same order as the intrinsic fluctuations of the
number of positive spins in the subsystem. Above this scale, the fluctuations
of the magnetization of the whole system dominate and lead to asymptotic separation from
the product measure.

\subsection{The critical window}
\label{subsec:critical-window}

Let $k=k(N)$ satisfy $\frac{k(N)}{N^{1/2}}\longrightarrow\alpha$
for some $\alpha>0$. Recall from \eqref{eq:hypergeometric-mixture} that $\mu_N\circ\mathcal P_k^{-1}=\mathbb E \big[\operatorname{Hyp}(N,NU_N,k)\big]$.
Let $\widetilde{\mathcal P}_{N,k}$ be a random variable such that,
conditionally on $U_N$,
$\widetilde{\mathcal P}_{N,k}\mid U_N\sim\bin (k,U_N)$.
Lemma~\ref{lem:hyp-bin} yields
$d_{\mathrm{TV}}\big(\mu_N\circ\mathcal P_k^{-1},
\mathcal L(\widetilde{\mathcal P}_{N,k})
\big)\leq\frac{4k }{N}$.
Since $k\sim\alpha N^{1/2}$, the right-hand side converges to zero. It
therefore suffices to study the binomial mixture
$\mathcal L(\widetilde{\mathcal P}_{N,k})$.

The following lemma identifies its fluctuations.

\begin{lemma}\label{lem:critical-mixture-limit}
Let $k=k(N)$ satisfy $k/N^{1/2}\to\alpha\in(0,\infty)$, and suppose that
\[
W_N=N^{1/4}\left(U_N-\frac{1}{2}\right)\Rightarrow X.
\]
Then
\[
\frac{\widetilde{\mathcal P}_{N,k}-k/2}{\sqrt{k}}\Rightarrow Z+\sqrt{\alpha}\,X,
\]
where $Z\sim\mathcal N(0,1/4)$ is independent of $X$.
\end{lemma}

\begin{proof}
Set
$R_N:=\frac{\widetilde{\mathcal P}_{N,k}-kU_N}{\sqrt{k}}$.
We first analyze the weak convergence of 
$(R_N,W_N)$ to $(Z,X)$ where $Z\sim\mathcal N(0,1/4)$ is independent of $X$.

For $t,s\in\mathbb R$, conditioning on $U_N$ gives
\[
\mathbb E\left[e^{itR_N+isW_N}\right]
=\mathbb E\left[e^{isW_N}\mathbb E\left[e^{itR_N}\,\middle|\,U_N\right]\right].
\]
For fixed $u\in[0,1]$, the conditional distribution of
$\widetilde{\mathcal P}_{N,k}$ given $U_N=u$ is $\operatorname{Bin}(k,u)$.
Therefore,
\[
\mathbb E\left[e^{itR_N}\,\middle|\,U_N=u\right]
=e^{-it\sqrt{k}u}\left(1-u+ue^{it/\sqrt{k}}\right)^k
=:g_{k,t}(u).
\]
Thus,
$
\mathbb E\big[e^{itR_N+isW_N}\big]
=\mathbb E\big[e^{isW_N}g_{k,t}(U_N)\big]$.

Writing $z:=it/\sqrt{k}$, we have
$\log g_{k,t}(u)=-it\sqrt{k}u+k\log\left(1+u(e^z-1)\right)$.

Since $e^z-1=z+\frac{z^2}{2}+O(k^{-3/2})$
and $\log(1+w)=w-\frac{w^2}{2}+O(|w|^3)$,
we obtain
\[
\begin{aligned}
k\log\left(1+u(e^z-1)\right)
&=ku\left(z+\frac{z^2}{2}\right)-\frac{k u^2z^2}{2}
+O\left((u+u^2+u^3)k^{-1/2}\right)\\
&=it\sqrt{k}\,u-\frac{t^2}{2}u(1-u)+O(k^{-1/2}),
\end{aligned}
\]
uniformly in $u\in[0,1]$ and therefore, 
$\log g_{k,t}(u)=-\frac{t^2}{2}u(1-u)+r_{k,t}(u)$,
where, for every fixed $t$,
$\sup_{u\in[0,1]}|r_{k,t}(u)|=O(k^{-1/2})$.
This yields
\[
g_{k,t}(u)
=\exp\left\{-\frac{t^2}{2}u(1-u)\right\}+o(1)
\]
uniformly in $u\in[0,1]$.

Since $(W_N)_N$ is tight and $U_N-\frac{1}{2}=N^{-1/4}W_N$,
we have $U_N\to 1/2$ in probability. It follows that
$g_{k,t}(U_N)\longrightarrow e^{-t^2/8}$
in probability. Moreover, $|g_{k,t}(U_N)|\leq1$, so the convergence also holds in $L^1$. Thus
\[
\mathbb E\left[e^{itR_N+isW_N}\right]
=e^{-t^2/8}\mathbb E\left[e^{isW_N}\right]+o(1)
\longrightarrow e^{-t^2/8}\mathbb E\left[e^{isX}\right].
\]
The limiting characteristic function factorizes into the characteristic
functions of $Z\sim\mathcal N(0,1/4)$ and $X$. Hence
$(R_N,W_N)\Rightarrow(Z,X)$,
where $Z$ and $X$ are independent.

Finally,
\[
\frac{\widetilde{\mathcal P}_{N,k}-k/2}{\sqrt{k}}
=R_N+\sqrt{k}\left(U_N-\frac{1}{2}\right)
=R_N+\frac{\sqrt{k}}{N^{1/4}}W_N.
\]
Since $\frac{\sqrt{k}}{N^{1/4}}\to\sqrt{\alpha}$,
the continuous mapping theorem yields
$\frac{\widetilde{\mathcal P}_{N,k}-k/2}{\sqrt{k}}
\Rightarrow Z+\sqrt{\alpha}\,X$.
\end{proof}

For the proofs of part~(2) and part~(3) of Theorem \ref{thm:main1} and for $b>0$, define
\[
A_N(b):=\Big\{\ell:\big|\ell-\frac{k}{2}\big|\geq b\sqrt{k}\Big\}.
\]

\begin{proof}[Proof of Theorem~\ref{thm:main1}, part~(2)]
Assume that $k/N^{1/2}\to\alpha\in(0,\infty)$ and set
\begin{equation}\label{Y_N-def}
Y_N:=\frac{\mathcal P_k-k/2}{\sqrt{k}}.
\end{equation}
By Lemma~\ref{lem:hyp-bin} and Lemma~\ref{lem:critical-mixture-limit},
$Y_N\Rightarrow Y_\alpha:=Z+\sqrt{\alpha}\,X$,
where $Z\sim\mathcal N(0,1/4)$ is independent of $X$.
Let $B_k\sim\operatorname{Bin}(k,1/2)$. By the central limit theorem,
\begin{equation}\label{B-k-def}
\frac{B_k-k/2}{\sqrt{k}}\Rightarrow Z.
\end{equation}
The distributions of $Y_\alpha$ and $Z$ are different, since
\[
\operatorname{Var}(Y_\alpha)=\frac14+\alpha\operatorname{Var}(X)>\frac14=\operatorname{Var}(Z).
\]
Consequently, since $Y_\alpha$ and $Z$ are centered, the distributions of $|Y_\alpha|$ and $|Z|$ cannot coincide.
The laws of $B_k$ are discrete and unimodal and the laws of $Y_N$ are discrete with the same support for $k=k(N)$, see \eqref{Y_N-def} and \eqref{B-k-def}, whereas $Y_\alpha$ and $Z$ both have densities with full support (i.e. the support is $\R$). Once we show that the laws of  $Y_N$ are unimodal (we refer to Corollary \ref{cor:unimodality-transition}), we can apply the following lemma.
\begin{lemma}\label{distances-unimod}
Let $(\nu_n)_{n \geq 1}$ and $(\tilde\nu_n)_{n \geq 1}$ be two sequences of probability measures such that for each $n$, $\nu_n$ and $\tilde\nu_n$ have the same finite support $S_n = \{y_1, \ldots, y_{b_n}\}$. Assume that all the 
$\nu_n$ and all the $\tilde\nu_n$ are unimodal. Assume further that the sequence $(\nu_n)_{n \geq 1}$ converges weakly to $\nu$ for $n \to \infty$ and 
the sequence $(\tilde\nu_n)_{n \geq 1}$ converges weakly to $\tilde\nu$ for $n \to \infty$ where $\nu$ and $\tilde\nu$ both have densities with full support.
Then, 
\begin{equation}\label{distance-conv}
\lim_{n\to\infty}d_{\mathrm{TV}}\big(\nu_n, \tilde\nu_n\big) = d_{\mathrm{TV}}(\nu, \tilde\nu)\, .
\end{equation}
\end{lemma}
The lemma implies that
\begin{equation}\label{formula-distance}
u(\alpha): = \lim_{N\to\infty}d_{\mathrm{TV}}\Big(\mu_N\circ\mathcal P_k^{-1},\Bin\big(k,\frac12\big)\Big) = d_{\mathrm{TV}}\Big(\nu_\alpha,\mathcal N(0,1/4)\Big)
\end{equation}
where $\nu_\alpha$ is the law of $Y_\alpha$.
Since $Y_\alpha$ and $Z$ have different laws, $u(\alpha) > 0$ and since both random variables have densities with full support, we also have $u(\alpha) < 1$.
\end{proof}
\begin{proof}[Proof of Lemma \ref{distances-unimod}]
Assume that
$(\eta_n)_{n \geq 1}$ and $(\tilde\eta_n)_{n \geq 1}$ are two sequences of probability measures such that for each $n$, $\eta_n$ and $\tilde \eta_n$ have densities with the same support and the densities of the $\eta_n$ and the $\tilde \eta_n$ are unimodal,
and the sequence $(\eta_n)_{n \geq 1}$ converges weakly to $\nu$ for $n \to \infty$ and 
the sequence $(\tilde\eta_n)_{n \geq 1}$ converges weakly to $\tilde\nu$ for $n \to \infty$ where $\nu$ and $\tilde\nu$ both have densities with full support. In this case, it is known that \eqref{distance-conv} holds true (with $d_{\mathrm{TV}}\big(\nu_n, \tilde\nu_n\big)$  replaced with $d_{\mathrm{TV}}\big(\eta_n, \tilde\eta_n\big)$ on the l.h.s.), see \cite{Hettmansperger-Klimko} and \cite{Walker}. The key point is that for probability measures with densities, the distance in total variation is given by the $L^1$-distance of the densities.
The discrete case follows by approximation: For each $n$, take $a_n< (1/2)\min\limits_{1\leq i \leq b_n-1}|s_{i+1} - s_i|$ and consider the probability measure $\eta_n$ with density $\nu_n(s_i)/(2a_n)$ on the intervals $[s_i - a_n, s_i + a_n]$ for  $1 \leq i \leq b_n$. Define, for each $n$, $\tilde\eta_n$ in the same way. Then the sequences
$(\eta_n)_{n \geq 1}$ and $(\tilde\eta_n)_{n \geq 1}$ satisfy the assumptions above and we have
$d_{\mathrm{TV}}\big(\nu_n, \tilde\nu_n\big) =d_{\mathrm{TV}}\big(\eta_n, \tilde\eta_n\big)$.
\end{proof}
\begin{proof}[Proof of Theorem~\ref{thm:main1}, part~(3)]
Assume that $k/N^{1/2}\to\infty$ and write
\[
Y_N:=\frac{\mathcal P_k-k/2}{\sqrt{k}}
=R_N+c_NW_N,
\]
where
\[
R_N:=\frac{\mathcal P_k-kU_N}{\sqrt{k}},
\qquad c_N:=\frac{\sqrt{k}}{N^{1/4}},
\qquad \text{and} \quad
W_N:=N^{1/4}\left(U_N-\frac12\right).
\]
Conditionally on $U_N$,
$\mathbb E[R_N\mid U_N]=0$
and
$\operatorname{Var}(R_N\mid U_N)
=U_N(1-U_N)\frac{N-k}{N-1}\leq\frac14$.

Hence $\mathbb E_{\mu_N}[R_N^2]\leq\frac14$,
so that $(R_N)_N$ is tight.
Moreover, $c_N\to\infty$, while $W_N\Rightarrow X$ and
$\mathbb P(X=0)=0$. It follows that
\[
c_N|W_N|\longrightarrow\infty
\qquad\text{in probability}.
\]
Indeed, for every $L>0$ and $\varepsilon>0$, for all sufficiently large $N$,
\[
\mathbb P(c_N|W_N|\leq L)\leq \mathbb P(|W_N|\leq\varepsilon),
\]
and therefore
\[
\limsup_{N\to\infty}\mathbb P(c_N|W_N|\leq L)
\leq \mathbb P(|X|\leq\varepsilon).
\]
Since this is true for all $\varepsilon$ the claim follows.

Now fix $b>0$. For every $M>0$,
\[
\mathbb P_{\mu_N}(|Y_N|\leq b)
\leq \mathbb P_{\mu_N}(|R_N|>M)
+\mathbb P_{\mu_N}(c_N|W_N|\leq b+M)
\]
by the triangle inequality. By Chebyshev's inequality,
\[
\mathbb P_{\mu_N}(|R_N|>M)\leq\frac{1}{4M^2},
\]
while $\mathbb P_{\mu_N}(c_N|W_N|\leq b+M)$ converges to zero. Thus
\[
\limsup_{N\to\infty}\mathbb P_{\mu_N}(|Y_N|\leq b)
\leq\frac{1}{4M^2}.
\]
Letting $M\to\infty$, we obtain
\[
\mathbb P_{\mu_N}(A_N(b))\longrightarrow1.
\]

On the other hand, if $B_k\sim\operatorname{Bin}(k,1/2)$, then the Central Limit Theorem gives
\[
\mathbb P_{\operatorname{Bin}(k,1/2)}(A_N(b))
\longrightarrow
\mathbb P(|Z|\geq b).
\]
Consequently, for every $b>0$,
\[
\liminf_{N\to\infty}
d_{\mathrm{TV}}\Big(\mu_N\circ\mathcal P_k^{-1},
\bin \Big(k,\frac12\Big)\Big)\geq 1-\mathbb P(|Z|\geq b).
\]
Finally, letting $b\to\infty$ yields
\[
d_{\mathrm{TV}}\Big(\mu_N\circ\mathcal P_k^{-1},
\operatorname{Bin}\Big(k,\frac12\Big)\Big)
\to 1.
\]
\end{proof}

\bibliographystyle{abbrv}
\bibliography{LiteraturDatenbank}    

@article {Hettmansperger-Klimko,
    AUTHOR = {Hettmansperger, Thomas P. and Klimko, Lawrence A.},
     TITLE = {A note on the strong convergence of distributions},
   JOURNAL = {Ann. Statist.},
  FJOURNAL = {The Annals of Statistics},
    VOLUME = {2},
      YEAR = {1974},
     PAGES = {597--598},
      ISSN = {0090-5364,2168-8966},
   MRCLASS = {62E20 (60F05)},
  MRNUMBER = {353541},
       URL =
              {http://links.jstor.org.tum-eaccess.de/sici?sici=0090-5364(197405)2:3<597:ANOTSC>2.0.CO;2-2&origin=MSN},
}

@article {Walker,
    AUTHOR = {Walker, Stephen G.},
     TITLE = {Comparing weak and strong convergence of density functions},
   JOURNAL = {Statist. Probab. Lett.},
  FJOURNAL = {Statistics \& Probability Letters},
    VOLUME = {200},
      YEAR = {2023},
     PAGES = {Paper No. 109878, 5},
      ISSN = {0167-7152,1879-2103},
   MRCLASS = {99-01},
  MRNUMBER = {4598015},
       DOI = {10.1016/j.spl.2023.109878},
       URL = {https://doi-org.tum-eaccess.de/10.1016/j.spl.2023.109878},
}

@article{DiaconisFreedman1980,
  author  = {Diaconis, Persi and Freedman, David},
  title   = {Finite Exchangeable Sequences},
  journal = {The Annals of Probability},
  volume  = {8},
  number  = {4},
  pages   = {745--764},
  year    = {1980},
  doi     = {10.1214/aop/1176994663}
}

@article{Ehm1991,
  author  = {Ehm, Werner},
  title   = {Binomial Approximation to the {P}oisson {B}inomial {D}istribution},
  journal = {Statistics \& Probability Letters},
  volume  = {11},
  number  = {1},
  pages   = {7--16},
  year    = {1991},
  doi     = {10.1016/0167-7152(91)90170-V}
}

@article {JKL25,
    AUTHOR = {Jalowy, Jonas and Kabluchko, Zakhar and L\"{o}we, Matthias},
     TITLE = {Propagation of chaos and residual dependence in {G}ibbs
              measures on finite sets},
   JOURNAL = {Math. Phys. Anal. Geom.},
  FJOURNAL = {Mathematical Physics, Analysis and Geometry. An International
              Journal Devoted to the Theory and Applications of Analysis and
              Geometry to Physics},
    VOLUME = {28},
      YEAR = {2025},
    NUMBER = {1},
     PAGES = {Paper No. 6, 25},
      ISSN = {1385-0172},
   MRCLASS = {82B05 (82B20)},
  MRNUMBER = {4879078},
       DOI = {10.1007/s11040-025-09503-5},
       URL = {https://doi.org/10.1007/s11040-025-09503-5},
}

@article{JKLM,
  title={When does the chaos in the {C}urie-{W}eiss model stop to propagate?},
  author={Jalowy, Jonas and Kabluchko, Zakhar and L{\"o}we, Matthias and Marynych, Alexander},
  journal={Electronic Journal of Probability},
  volume={28},
  pages={1--17},
  year={2023},
  publisher={The Institute of Mathematical Statistics and the Bernoulli Society}
}

@article {BAZ_chaos,
    AUTHOR = {Ben Arous, G. and Zeitouni, O.},
     TITLE = {Increasing propagation of chaos for mean field models},
   JOURNAL = {Ann. Inst. H. Poincar\'{e} Probab. Statist.},
  FJOURNAL = {Annales de l'Institut Henri Poincar\'{e}. Probabilit\'{e}s et
              Statistiques},
    VOLUME = {35},
      YEAR = {1999},
    NUMBER = {1},
     PAGES = {85--102},
      ISSN = {0246-0203},
   MRCLASS = {60F10 (60K35)},
  MRNUMBER = {1669916},
MRREVIEWER = {Hans-Otto Georgii},
       DOI = {10.1016/S0246-0203(99)80006-5},
       URL = {https://doi.org/10.1016/S0246-0203(99)80006-5},
}

@article {EllisNewman_80,
    AUTHOR = {Ellis, Richard S. and Newman, Charles M. and Rosen, Jay S.},
     TITLE = {Limit theorems for sums of dependent random variables
              occurring in statistical mechanics. {II}. {C}onditioning,
              multiple phases, and metastability},
   JOURNAL = {Z. Wahrsch. Verw. Gebiete},
  FJOURNAL = {Zeitschrift f\"{u}r Wahrscheinlichkeitstheorie und Verwandte
              Gebiete},
    VOLUME = {51},
      YEAR = {1980},
    NUMBER = {2},
     PAGES = {153--169},
      ISSN = {0044-3719},
   MRCLASS = {82A05 (60F05)},
  MRNUMBER = {566313},
MRREVIEWER = {D. Sz\'{a}sz},
       DOI = {10.1007/BF00536186},
       URL = {https://doi.org/10.1007/BF00536186},
}

@article {EL10,
    AUTHOR = {Eichelsbacher, Peter and L{\"o}we, Matthias},
     TITLE = {Stein's method for dependent random variables occurring in
              statistical mechanics},
   JOURNAL = {Electron. J. Probab.},
  FJOURNAL = {Electronic Journal of Probability},
    VOLUME = {15},
      YEAR = {2010},
     PAGES = {no. 30, 962--988},
      ISSN = {1083-6489},
   MRCLASS = {60F05 (82B20)},
  MRNUMBER = {2659754 (2011g:60041)},
MRREVIEWER = {Zhong Gen Su},
       DOI = {10.1214/EJP.v15-777},
       URL = {http://dx.doi.org/10.1214/EJP.v15-777},
}

@article {Chatterjee_Shao,
    AUTHOR = {Chatterjee, Sourav and Shao, Qi-Man},
     TITLE = {Nonnormal approximation by {S}tein's method of exchangeable
              pairs with application to the {C}urie-{W}eiss model},
   JOURNAL = {Ann. Appl. Probab.},
  FJOURNAL = {The Annals of Applied Probability},
    VOLUME = {21},
      YEAR = {2011},
    NUMBER = {2},
     PAGES = {464--483},
      ISSN = {1050-5164},
   MRCLASS = {60F05 (60G09)},
  MRNUMBER = {2807964 (2012b:60102)},
MRREVIEWER = {Irene Crimaldi},
       DOI = {10.1214/10-AAP712},
       URL = {http://dx.doi.org/10.1214/10-AAP712},
}

@book{BovierSMoDS,
author="Bovier, Anton",
title="{Statistical Mechanics of Disordered Systems - A Mathematical Perspective}",
publisher="Cambridge Series in Statistical and Probabilistic Mathematics",
year="2006",
isbn="0-521-84991-8",
}

@article {GHS70,
    AUTHOR = {Griffiths, Robert B. and Hurst, C. A. and Sherman, S.},
     TITLE = {Concavity of magnetization of an {I}sing ferromagnet in a
              positive external field},
   JOURNAL = {J. Mathematical Phys.},
  FJOURNAL = {Journal of Mathematical Physics},
    VOLUME = {11},
      YEAR = {1970},
     PAGES = {790--795},
      ISSN = {0022-2488,1089-7658},
   MRCLASS = {81.05},
  MRNUMBER = {266507},
       DOI = {10.1063/1.1665211},
       URL = {https://doi-org.tum-eaccess.de/10.1063/1.1665211},
}

@book {Velenik_book,
    AUTHOR = {Friedli, S. and Velenik, Y.},
     TITLE = {Statistical mechanics of lattice systems},
 PUBLISHER = {Cambridge University Press, Cambridge},
      YEAR = {2018},
     PAGES = {xix+622},
      ISBN = {978-1-107-18482-4},
   MRCLASS = {82-01 (82B05)},
  MRNUMBER = {3752129},
}

@book {EllisEntropyLargeDeviationsAndStatisticalMechanics,
    AUTHOR = {Ellis, Richard S.},
     TITLE = {Entropy, large deviations, and statistical mechanics},
    SERIES = {Classics in Mathematics},
      NOTE = {Reprint of the 1985 original},
 PUBLISHER = {Springer-Verlag},
   ADDRESS = {Berlin},
      YEAR = {2006},
     PAGES = {xiv+364},
      ISBN = {978-3-540-29059-9; 3-540-29059-1},
   MRCLASS = {82-02 (60F10 60K35 82B05)},
  MRNUMBER = {2189669 (2006m:82002)},
}

\appendix
\section{Unimodality of the law of the number of positive spins among the first $k=k(N)$ spins}
For the application of Lemma \ref{distances-unimod} we need to analyze 
the distribution of the number of positive spins in a subsystem, in particular we need to know whether it is unimodal. Define
\[
p_{N,k}(j):=\mu_N(\mathcal P_k=j),
\qquad j=0,\ldots,k,
\]
\[
\rho_{N,k}(dx):= \frac{1}{Z_{N,k}} e^{-Nx^2/2}(2\cosh x)^{N-k}\,dx.
\]
where $Z_{N,k}$ is a normalization factor. For $k$ is odd, we also introduce the probability measure
\[
\widehat\rho_{N,k}(dx):=
\frac{\cosh x}{\mathbb E_{\rho_{N,k}}[\cosh X]}\rho_{N,k}(dx).
\]

\begin{proposition}\label{prop:unimodality-criterion}
For every $N$ and $k\leq N$, the following is true
	
\begin{enumerate}
\item If $k$ is even, then $(p_{N,k}(j))_{j=0}^k$ is log-concave, and hence
unimodal, if and only if
\begin{equation}\label{eq:sinh^2gerade}
k\,\mathbb E_{\rho_{N,k}}[\sinh^2X]\le 1.
\end{equation}
If the reverse strict inequality holds, then the midpoint $k/2$ is a strict
local minimum and the distribution is not unimodal.
		
\item If $k$ is odd, then $(p_{N,k}(j))_{j=0}^k$ is log-concave, and hence
unimodal, if and only if
\begin{equation}\label{eq:sinh^2ungerade}
(k-1)\mathbb E_{\widehat\rho_{N,k}}[\sinh^2X]\le 1.
\end{equation}
If the reverse strict inequality holds, then the two midpoints $(k-1)/2$ and $(k+1)/2$ are a
local minimum and the distribution is not unimodal.
\end{enumerate}
\end{proposition}
\begin{proof}
To derive a suitable representation of $p_{N,k}$, recall that at
$\beta=1$ and $h=0$ the Curie--Weiss measure is given by
\[
\mu_N(\sigma)=\frac{1}{Z_N}\exp\left\{\frac{1}{2N}\left(\sum_{i=1}^N\sigma_i\right)^2\right\}.
\]
We use the well known Hubbard--Stratonovich identity
\[
\exp\left\{\frac{S^2}{2N}\right\}
=\sqrt{\frac{N}{2\pi}}\int_{\mathbb R} \exp\left\{-\frac{N}{2}x^2+xS\right\}\,dx,
\]
(which is simply the Gaussian integral obtained by completing the square.)
Recall that $\mathcal P_k$ is the number of positive spins among the first $k$ spins and fix $\mathcal P_k=j$. There are $\binom{k}{j}$ configurations of the
first $k$ spins with exactly $j$ positive spins, and for each of them $\sum_{i=1}^k\sigma_i=j-(k-j)=2j-k$.
Hence
\[
p_{N,k}(j)=\mu_N(\mathcal P_k=j)=
\frac{1}{Z_N}\binom{k}{j}\sum_{\sigma_{k+1},\ldots,\sigma_N}
\exp\Big\{\frac{1}{2N}\Big(2j-k+\sum_{i=k+1}^N\sigma_i\Big)^2\Big\}.
\]
Applying the Hubbard--Stratonovich identity with
$S=2j-k+\sum_{i=k+1}^N\sigma_i$
gives
\[
p_{N,k}(j)=\frac{1}{Z_N}\sqrt{\frac{N}{2\pi}}\binom{k}{j}
\int_{\mathbb R}e^{-Nx^2/2}e^{(2j-k)x}\sum_{\sigma_{k+1},\ldots,\sigma_N}
e^{x\sum_{i=k+1}^N\sigma_i}\,dx.
\]
The remaining spin sum factorizes:
\[
\sum_{\sigma_{k+1},\ldots,\sigma_N} e^{x\sum_{i=k+1}^N\sigma_i}
=\prod_{i=k+1}^N\sum_{\sigma_i\in\{-1,1\}}e^{x\sigma_i}
=(2\cosh x)^{N-k}.
\]
Hence, 
\[
p_{N,k}(j)=C_N\binom{k}{j}\int_{\mathbb R}
e^{-Nx^2/2}(2\cosh x)^{N-k}e^{(2j-k)x}\,dx,
\]
where
\[
C_N:=\frac{1}{Z_N}\sqrt{\frac{N}{2\pi}}
\]
does not depend on $j$. Thus, defining
\[
M_{N,k}(t):=\int_{\mathbb R} e^{-Nx^2/2}(2\cosh x)^{N-k}e^{tx}\,dx,
\]
we obtain the representation
\begin{equation}\label{represent-weights}
p_{N,k}(j)= C_N\binom{k}{j}M_{N,k}(2j-k).
\end{equation}
For convenience define
\[\psi(t):=\log M_{N,k}(t)\] and $t_j:=2j-k$.
Thus we obtain
\begin{multline*}
\log p_{N,k}(j+1)+\log p_{N,k}(j-1)-2\log p_{N,k}(j)
\\= \log\frac{k^2-t_j^2}{(k+2)^2-t_j^2} +\psi(t_j+2)+\psi(t_j-2)-2\psi(t_j).
\end{multline*}
We want to show that both terms on the right-hand side are even and
non-increasing as functions of $|t_j|$.
For the first term this is immediate, since
\[
t\longmapsto \log\frac{k^2-t^2}{(k+2)^2-t^2}
\]
is even and strictly decreasing for $t>0$.
Thus it remains to show that
\[
t\longmapsto \psi(t+2)+\psi(t-2)-2\psi(t)
\]
is even and non-increasing for $t\geq0$. To this end put $n=N-k$ and define, for $\sigma \in \{-1,1\}^n$,
\[
\mu_h(\sigma) = \mu_{n,N, h}(\sigma) = \frac{1}{\mathcal Z_{n,N}(h)}
\exp\left\{\frac{1}{2N}\left(\sum_{i=1}^n\sigma_i\right)^2+h\sum_{i=1}^n\sigma_i
\right\},
\]
with normalization factor
\[
\mathcal Z_{n,N}(h):=\sum_{\sigma\in\{-1,1\}^n}
\exp\left\{\frac{1}{2N}\left(\sum_{i=1}^n\sigma_i\right)^2+h\sum_{i=1}^n\sigma_i
\right\}.
\]
Writing $S_n=\sum_{i=1}^n\sigma_i$ and expanding
$(2\cosh x)^n=\sum_{\sigma\in\{-1,1\}^n}e^{xS_n}$,
we may apply the Gaussian integral underlying the Hubbard--Stratonovich
transformation once more. This yields
\[
M_{N,k}(t)=\sum_{\sigma\in\{-1,1\}^n}\int_{\mathbb R}\exp\left\{-\frac N2x^2+x(S_n+t)\right\}\,dx=\sqrt{\frac{2\pi}{N}}\,e^{t^2/(2N)}
\mathcal Z_{n,N}\left(\frac{t}{N}\right).
\]
Hence
\begin{equation}\label{psi-formula}
\psi(t)
=
\frac12\log\frac{2\pi}{N}
+\frac{t^2}{2N}
+\log\mathcal Z_{n,N}\left(\frac{t}{N}\right).
\end{equation}
We write $\E_h[\cdot]$ and $\operatorname{Var}_h(\cdot)$ for the expectation and the variance with respect to $\mu_h$ and we will use the following well-know identities, with $S_n=\sum_{i=1}^n\sigma_i$,
\[
\frac{d}{dh}\log\mathcal Z_{n,N}(h)=\E_h[S_n],
\qquad \frac{d^2}{dh^2}\log\mathcal Z_{n,N}(h) = \operatorname{Var}_h(S_n)\, .
\]
Differentiating \eqref{psi-formula}, we obtain
\[
\psi'(t)=\frac{t}{N}+\frac{1}{N}\E_{t/N}[S_n]
\]
and, differentiating once more,
\begin{equation}\label{psi-doubleprime}
\psi''(t)=\frac1N+\frac{1}{N^2}\operatorname{Var}_{t/N}(S_n).
\end{equation}
By spin-flip symmetry,
\[
\mathcal Z_{n,N}(h)=\mathcal Z_{n,N}(-h),
\]
and hence $\psi''$ is even.

To obtain the required monotonicity, we use the GHS inequality
\cite{GHS70}. Notice that the interaction defining $\mathcal Z_{n,N}$ is
ferromagnetic, since,
\[
\frac{1}{2N}S_n^2 =\frac{1}{N}\sum_{1\leq i<j\le n}\sigma_i\sigma_j+\frac{n}{2N}
\]
For a ferromagnetic Ising model with non-negative external field, the GHS
inequality states that the truncated three-point correlations (or third Ursell functions) are non-positive:
\[
\langle\sigma_i;\sigma_j;\sigma_\ell\rangle_h\leq0,
\]
where
\[
\begin{aligned}
	\langle\sigma_i;\sigma_j;\sigma_\ell\rangle_h
	&:= \E_h[\sigma_i\sigma_j\sigma_\ell] -\E_h[\sigma_i\sigma_j]\E_h[\sigma_\ell]\\
	&\quad-\E_h[\sigma_i\sigma_\ell]\E_h[\sigma_j]
	-\E_h[\sigma_j\sigma_\ell]\E_h[\sigma_i]
+2\E_h[\sigma_i]\E_h[\sigma_j]\E_h[\sigma_\ell].
\end{aligned}
\]
To see how this enters, recall that
\[
\frac{d^2}{dh^2}\log\mathcal Z_{n,N}(h)=
\operatorname{Var}_h(S_n)=\sum_{i,j=1}^n\left(
\E_h[\sigma_i\sigma_j]-\E_h[\sigma_i]\E_h[\sigma_j]\right).
\]
Again differentiating with respect to $h$ gives
\[
\begin{aligned}
\frac{d^3}{dh^3}\log\mathcal Z_{n,N}(h)
&=\sum_{i,j=1}^n\frac{d}{dh}\left(\E_h[\sigma_i\sigma_j]
-\E_h[\sigma_i]\E_h[\sigma_j]\right)\\
&= \sum_{i,j,\ell=1}^n \langle\sigma_i;\sigma_j;\sigma_\ell\rangle_h.
\end{aligned}
\]
Hence GHS implies, for $h\geq0$,
$\frac{d^3}{dh^3}\log\mathcal Z_{n,N}(h)\le 0$.
Equivalently,
$h\mapsto\operatorname{Var}_h(S_n)$ is non-increasing on $[0,\infty)$. Recalling \eqref{psi-doubleprime},
we conclude that $\psi''$ is even and non-increasing on $[0,\infty)$.

Hence, for $t\ge 0$ we have
\[
\psi'(t+2)+\psi'(t-2)-2\psi'(t)=
\int_0^2 \left(\psi''(t+s)-\psi''(t-s)\right)\,ds \le 0,
\]
since $|t+s|\geq|t-s|$.
Hence $\psi(t+2)+\psi(t-2)-2\psi(t)$ is even and non-increasing on $[0,\infty)$.

Hence we obtain that
$\log p_{N,k}(j+1)+\log p_{N,k}(j-1)-2\log p_{N,k}(j)$
is maximal at those indices for which $|t_j|=|2j-k|$ is minimal.

Now suppose first that $k=2r$ is even. Since the discrete second difference of $\log p_{N,k}$ is maximal at $j=r$, we have, for every $j$,
\[
\log p_{N,k}(j+1)+\log p_{N,k}(j-1)-2\log p_{N,k}(j)
\]
\[
\le 
\log p_{N,k}(r+1)+\log p_{N,k}(r-1)-2\log p_{N,k}(r).
\]
By symmetry, $p_{N,k}(r-1)=p_{N,k}(r+1)$,
and hence
\[
\log p_{N,k}(r+1)+\log p_{N,k}(r-1)-2\log p_{N,k}(r)
= 2\log\frac{p_{N,k}(r+1)}{p_{N,k}(r)}.
\]
Therefore, if 
\begin{equation}\label{eq:central_ineq}
p_{N,k}(r+1)\leq p_{N,k}(r),
\end{equation} 
all discrete second differences of $\log p_{N,k}$ are non-positive. Thus
$(p_{N,k}(j))_{j=0}^k$ is log-concave, and hence unimodal.
If instead $p_{N,k}(r+1)>p_{N,k}(r)$,
then by symmetry
\[
p_{N,k}(r-1)=p_{N,k}(r+1)>p_{N,k}(r),
\]
so that $r$ is a strict local minimum and the distribution is not unimodal.
Let us see, when \eqref{eq:central_ineq} is satisfied: 
Recalling \eqref{represent-weights},
\[
\frac{p_{N,k}(r+1)}{p_{N,k}(r)}=\frac{r}{r+1}\,\frac{M_{N,k}(2)}{M_{N,k}(0)}
\]
and $\rho_{N,k}$ is symmetric, we obtain
\[
\frac{M_{N,k}(2)}{M_{N,k}(0)}=\mathbb E_{\rho_{N,k}}[e^{2X}]=
\mathbb E_{\rho_{N,k}}[\cosh(2X)]=1+2\mathbb E_{\rho_{N,k}}[\sinh^2X]\, .
\]
Hence
\[
\frac{p_{N,k}(r+1)}{p_{N,k}(r)}
=\frac{r}{r+1}\left(1+2\mathbb E_{\rho_{N,k}}[\sinh^2X]\right).
\]
Therefore,
$\frac{p_{N,k}(r+1)}{p_{N,k}(r)}\le 1$
is equivalent to
\[
r\left(1+2\mathbb E_{\rho_{N,k}}[\sinh^2X]\right)\leq r+1,
\]
or, since $k=2r$,
\[
k\,\mathbb E_{\rho_{N,k}}[\sinh^2X]\le 1.
\]
The case of odd $k$ is analogous. More precisely, let $k=2r+1$. By symmetry,
\[
p_{N,k}(r)=p_{N,k}(r+1), \qquad \text{as well as } \quad p_{N,k}(r-1)=p_{N,k}(r+2).
\]
Since now $|t_j|=|2j-k|$ is minimal at $j=r,r+1$, the preceding
monotonicity argument shows that $(p_{N,k}(j))_{j=0}^k$ is log-concave
if and only if
$p_{N,k}(r+2)\leq p_{N,k}(r+1)$.
Moreover,
\[
\frac{p_{N,k}(r+2)}{p_{N,k}(r+1)}
=
\frac{r}{r+2}\frac{M_{N,k}(3)}{M_{N,k}(1)}.
\]
To find an equivalent formulation, recall that
$\widehat\rho_{N,k}(dx):=
\frac{\cosh x}{\mathbb E_{\rho_{N,k}}[\cosh X]}\rho_{N,k}(dx)$.

Using the symmetry of $\rho_{N,k}$ and
$\frac{\cosh(3x)}{\cosh x}=1+4\sinh^2x$,
we obtain
\[
\frac{M_{N,k}(3)}{M_{N,k}(1)}=1+4\mathbb E_{\widehat\rho_{N,k}}[\sinh^2X].
\]
This yields
\[
\frac{p_{N,k}(r+2)}{p_{N,k}(r+1)}\leq1\quad \text{if and only if} \quad 
(k-1)\mathbb E_{\widehat\rho_{N,k}}[\sinh^2X]\leq1.
\]
If the reverse strict inequality holds, then
\[
p_{N,k}(r-1)=p_{N,k}(r+2)>p_{N,k}(r)=p_{N,k}(r+1),
\]
so the distribution is not unimodal.
\end{proof}

\begin{corollary}\label{cor:unimodality-transition}
For every $\varepsilon>0$ there exists $N_\varepsilon<\infty$ such that, for all $N\geq N_\varepsilon$, the following statements hold:
\begin{enumerate}
\item If $k\leq \left(\frac12-\varepsilon\right)N$,
then the distribution $(p_{N,k}(j))_{j=0}^k$ is log-concave and hence
unimodal.
\item If $k\geq \left(\frac12+\varepsilon\right)N$,
then the distribution $(p_{N,k}(j))_{j=0}^k$ is not unimodal.
\end{enumerate}
\end{corollary}

\begin{proof}
By Proposition~\ref{prop:unimodality-criterion}, it suffices to analyze
\eqref{eq:sinh^2gerade} and \eqref{eq:sinh^2ungerade}, respectively.

There are three relevant regimes. Assume first that $k/\sqrt N$ remains
bounded, more precisely that
$\frac{k}{\sqrt N}\longrightarrow\alpha\in[0,\infty)$.
Let $X$ be a random variable with law $\rho_{N,k}$, then the density of $X$ is proportional to
\[
\exp\left\{-\frac N2x^2+(N-k)\log\cosh x\right\}.
\]
Writing $q(x):=\frac{x^2}{2}-\log\cosh x$,
the exponent can be rewritten in the form of
$-\frac{k}{2}x^2-(N-k)q(x)$.
Since
\[
q(x)=\frac{x^4}{12}+O(x^6)
\qquad\text{as }x\to0,
\]
the quadratic and quartic terms are of the same order when
$x$ is of order $N^{-1/4}$. We therefore set
\[
V_N:=N^{1/4}X \qquad \text{and} \qquad \alpha_N:=\frac{k}{\sqrt N}.
\]
Up to normalization, the density of $V_N$ is then
\[
\exp\left\{-\frac{\alpha_N}{2}v^2-(N-k)q(N^{-1/4}v)\right\}.
\]
For every fixed $v$,
\[
(N-k)q(N^{-1/4}v)=\frac{N-k}{N}\frac{v^4}{12}+O(N^{-1/2}v^6)\longrightarrow \frac{v^4}{12},
\]
and hence
\[
\frac{\alpha_N}{2}v^2+(N-k)q(N^{-1/4}v)\longrightarrow\frac{\alpha}{2}v^2+\frac{v^4}{12}.
\]
To analyze the behavior of $V_N$, note that there exists $c>0$ such that
\[
q(x)\geq cx^4\quad\text{for }|x|\leq1,
\qquad \text{and} \qquad 
q(x)\geq cx^2\quad\text{for }|x|\geq1.
\]
Since $k/\sqrt N$ is bounded, $(N-k)/N$ is bounded away from zero for
large $N$. These estimates give an integrable bound, independent of $N$,
for the rescaled densities. Dominated convergence therefore also gives
convergence of the normalizing constants. Therefore,
$V_N\Rightarrow V_\alpha$,
where $V_\alpha$ has density proportional to
\[
\exp\left\{-\frac{\alpha}{2}v^2-\frac{v^4}{12}\right\}.
\]
The preceding tail bounds imply uniform integrability of all polynomial
powers of $V_N$. Hence the weak convergence $V_N\Rightarrow V_\alpha$
also implies convergence of all polynomial moments; in particular,
\[
\mathbb E[V_N^2]\longrightarrow\mathbb E[V_\alpha^2],
\qquad \text{and}\quad 
\mathbb E[V_N^4]\longrightarrow\mathbb E[V_\alpha^4].
\]

Since $\sinh^2x=x^2+O(x^4)$,
we have
\[
k\mathbb E_{\rho_{N,k}}[\sinh^2X]=\frac{k}{\sqrt N}\mathbb E[V_N^2]
+O\left(\frac{k}{N}\mathbb E[V_N^4]\right)\longrightarrow\alpha\mathbb E[V_\alpha^2].
\]
Now, the right hand side is strictly smaller than $1$. Indeed, write
$\Phi_\alpha(v):=\frac{\alpha}{2}v^2+\frac{v^4}{12}$.
Since $\Phi_\alpha'(v)=\alpha v+\frac{v^3}{3}$,
we have
\[
\frac{d}{dv}\left(v e^{-\Phi_\alpha(v)}\right)
=\left(1-\alpha v^2-\frac{v^4}{3}\right)e^{-\Phi_\alpha(v)}.
\]
Hence, by integration,
\[
0=\int_{\mathbb R}\left(1-\alpha v^2-\frac{v^4}{3}\right)e^{-\Phi_\alpha(v)}\,dv.
\]
Dividing by the normalizing constant of $V_\alpha$ yields
\[
\alpha\mathbb E[V_\alpha^2]+\frac13\mathbb E[V_\alpha^4]=1.
\]
Hence $\alpha\mathbb E[V_\alpha^2]<1$
for $\alpha>0$. (The assertion is immediate for $\alpha=0$).

For odd $k$, the measure $\widehat\rho_{N,k}$ differs from $\rho_{N,k}$
only by the factor $\frac{\cosh X}{\mathbb E_{\rho_{N,k}}[\cosh X]}$ in the density.
Since $X=N^{-1/4}V_N$ and
\[
\cosh X=1+O(N^{-1/2}V_N^2),
\]
this tilt is asymptotically negligible. Therefore,
\[
(k-1)\mathbb E_{\widehat\rho_{N,k}}[\sinh^2X]
\longrightarrow
\alpha\mathbb E[V_\alpha^2]<1
\]
as well.

Let us turn to the regime $\sqrt N\ll k\ll N$. In this case the quadratic
term in the exponent dominates. Indeed, it suggests the scale
$x\asymp k^{-1/2}$, and on this scale the quartic contribution is of order
$Nx^4\asymp\frac{N}{k^2}\to 0$, since $k\gg\sqrt N$. We therefore set
$Y:=\sqrt{k}\,X$.
The density of $Y$ under $\rho_{N,k}$ is proportional to
\[
\exp\left\{-\frac{y^2}{2}-\frac{N-k}{12k^2}y^4+O\left(\frac{N}{k^3}y^6\right)\right\}.
\]
Set
\[
a_N:=\frac{N-k}{12k^2}.
\]
Then $a_N\to0$, while $N/k^3=o(a_N)$. Thus, up to an error of smaller
order, the law of $Y$ is a quartic perturbation of the standard Gaussian
law. If $Z\sim\mathcal N(0,1)$, then
\[
\mathbb E_{\rho_{N,k}}[Y^2]=\frac{\mathbb E\left[Z^2e^{-a_NZ^4}\right]}{\mathbb E\left[e^{-a_NZ^4}\right]}+o(a_N).
\]
Expanding numerator and denominator to first order in $a_N$ gives
\[
\begin{aligned}
\mathbb E_{\rho_{N,k}}[Y^2]&=\frac{\mathbb E[Z^2]-a_N\mathbb E[Z^6]+o(a_N)}{1-a_N\mathbb E[Z^4]+o(a_N)}\\
&=\mathbb E[Z^2]-a_N\left(\mathbb E[Z^6]-\mathbb E[Z^2]\mathbb E[Z^4]\right)+o(a_N).
\end{aligned}
\]
Using
\[
\mathbb E[Z^2]=1,\qquad
\mathbb E[Z^4]=3,\qquad \text{and} \qquad 
\mathbb E[Z^6]=15,
\]
we obtain
\[
\mathbb E_{\rho_{N,k}}[Y^2]=1-\frac{N-k}{k^2}+o\left(\frac{N}{k^2}\right).
\]
Moreover,
\[
\mathbb E_{\rho_{N,k}}[Y^4]=3+o(1).
\]
Since $X=Y/\sqrt{k}$ and
\[
\sinh^2x=x^2+\frac{x^4}{3}+O(x^6),
\]
we therefore obtain
\[
\begin{aligned}
k\,\mathbb E_{\rho_{N,k}}[\sinh^2X]&=\mathbb E_{\rho_{N,k}}[Y^2]
+\frac{1}{3k}\mathbb E_{\rho_{N,k}}[Y^4]+o\left(\frac{N}{k^2}\right)\\
&=1-\frac{N-k}{k^2}+\frac1k+o\left(\frac{N}{k^2}\right)\\
&=1+\frac{2k-N}{k^2}+o\left(\frac{N}{k^2}\right).
\end{aligned}
\]
Therefore,
\[
k\,\mathbb E_{\rho_{N,k}}[\sinh^2X]-1
=\frac{2k-N}{k^2}+o\left(\frac{N}{k^2}\right)=-\frac{N}{k^2}(1+o(1)),
\]
since $k=o(N)$, and hence the right-hand side is negative for all sufficiently large $N$.

For odd $k$, recall the form of $\widehat\rho_{N,k}$.
As $X=O_{\mathbb P}(k^{-1/2})$, the additional factor satisfies
$\cosh X=1+O_{\mathbb P}(k^{-1})$ and is negligible on the scale
$N/k^2$. The same expansion therefore yields
\[
(k-1)\mathbb E_{\widehat\rho_{N,k}}[\sinh^2X]-1
=
\frac{2k-N}{k^2}
+o\left(\frac{N}{k^2}\right),
\]
which is again negative when $k=o(N)$.

Finally, suppose that $\frac{k}{N}\longrightarrow c\in(0,1]$.
Writing $c_N=k/N$, the density of $\rho_{N,k}$ is proportional to
$e^{-N\Phi_{c_N}(x)}$, where
\[
\Phi_c(x)=\frac{x^2}{2}-(1-c)\log\cosh x=
\frac c2x^2+\frac{1-c}{12}x^4+O(x^6).
\]
Since $c_N\to c>0$, the quadratic coefficient stays bounded away from
zero. Thus the leading term in the exponent is
$-\frac{Nc_N}{2}x^2$,
which shows that the mass of $\rho_{N,k}$ is concentrated on the scale
$x\asymp N^{-1/2}$. Equivalently, $\sqrt N\,X$ is of order one.
Set $Y:=\sqrt{c_NN}\,X$.
Then
\[
N\Phi_{c_N}\left(\frac{y}{\sqrt{c_NN}}\right)
=\frac{y^2}{2}+\frac{1-c_N}{12c_N^2N}y^4+O\left(\frac{y^6}{N^2}\right).
\]
Thus, writing $a_N:=\frac{1-c_N}{12c_N^2N}$,
the law of $Y$ is a quartic perturbation of the standard Gaussian law and as above,
\[
\begin{aligned}
\mathbb E_{\rho_{N,k}}[Y^2]&=\mathbb E[Z^2]-a_N\left(\mathbb E[Z^6]
-\mathbb E[Z^2]\mathbb E[Z^4]\right)+O(N^{-2})\\
&=1-\frac{1-c_N}{c_N^2N}+O(N^{-2}),
\end{aligned}
\]
where $Z\sim\mathcal N(0,1)$. Thus,
\[
\mathbb E_{\rho_{N,k}}[X^2]=\frac{1}{c_NN}+\frac{c_N-1}{c_N^3N^2}
+O(N^{-3}).
\]
Similarly we see,
\[
\mathbb E_{\rho_{N,k}}[Y^4]=3+O(N^{-1}),
\]
and hence
\[
\mathbb E_{\rho_{N,k}}[X^4]=\frac{3}{c_N^2N^2}+O(N^{-3}).
\]
Using again $\sinh^2x=x^2+\frac{x^4}{3}+O(x^6)$,
we find
\[\mathbb E_{\rho_{N,k}}[\sinh^2X]=\frac{1}{c_NN}+\frac{2c_N-1}{c_N^3N^2}
+O(N^{-3}),\]
and hence, for even $k$,
\[
k\,\mathbb E_{\rho_{N,k}}[\sinh^2X]-1=\frac{2c_N-1}{c_N^2N}
+O(N^{-2}).
\]
For odd $k$, expanding the additional factor $\cosh X$ gives
$\mathbb E_{\widehat\rho_{N,k}}[\sinh^2X]=
\frac{1}{c_NN}+\frac{3c_N-1}{c_N^3N^2}+O(N^{-3})$,
and therefore
\[
(k-1)\mathbb E_{\widehat\rho_{N,k}}[\sinh^2X]-1
=\frac{2c_N-1}{c_N^2N}+O(N^{-2}).
\]
Thus, for both parities, the sign is eventually negative if $c<1/2$
and positive if $c>1/2$.

Together with the preceding case $k=O(\sqrt N)$, these estimates cover
all possible subsequences. Hence, uniformly away from $k=N/2$, the
distribution is log-concave and unimodal below $N/2$, and non-unimodal
above $N/2$.
\end{proof}

\end{document}